\documentclass[pdflatex,sn-mathphys-num]{sn-jnl}

\usepackage{graphicx}%
\usepackage{multirow}%
\usepackage{amsmath,amssymb,amsfonts}%
\usepackage{amsthm}%
\usepackage{mathrsfs}%
\usepackage[title]{appendix}%
\usepackage{xcolor}%
\usepackage{textcomp}%
\usepackage{manyfoot}%
\usepackage{booktabs}%
\usepackage{algorithm}%
\usepackage{algorithmicx}%
\usepackage{algpseudocode}%
\usepackage{enumitem}%

\newcommand{\D}{\mathbb D}

\newcommand{\C}{\mathbb C}

\newcommand{\Ocal}{\mathcal O}
\newcommand{\Mcal}{\mathcal M}
\newcommand{\Ical}{\mathcal I}

\newcommand{\Ccal}{\mathcal C}
\newcommand{\dd}{\,\mathrm d}
\newcommand{\dist}{\operatorname{dist}}

\newcommand{\ii}{\mathrm{i}}

\theoremstyle{thmstyleone}%
\newtheorem{theorem}{Theorem}
\newtheorem{proposition}[theorem]{Proposition}
\newtheorem{lemma}[theorem]{Lemma}
\newtheorem{corollary}[theorem]{Corollary}

\theoremstyle{thmstyletwo}%

\newtheorem{remark}[theorem]{Remark}

\theoremstyle{thmstylethree}%
\newtheorem{definition}[theorem]{Definition}

\begin{document}

\title[Contour-count indicator fields]{Contour-count indicator fields for visible pole clusters in meromorphic continuation}
\author[1]{\fnm{Zhiliang} \sur{Deng}}
\author[2]{\fnm{Xiaomei} \sur{Yang}}

\affil[1]{\orgdiv{School of Mathematical Science}, \orgname{University of Electronic Science and Technology of China}, \orgaddress{\city{Chengdu}, \country{China}}}
\affil[2]{\orgdiv{School of Mathematics}, \orgname{Southwest Jiaotong University}, \orgaddress{\city{Chengdu}, \country{China}}}

\email{dengzhl@uestc.edu.cn}
\email{yangxiaomath@swjtu.edu.cn}

\abstract{We develop a contour-count indicator method for visible pole clusters in outward meromorphic continuation from circular boundary data.  The method starts from determinant characteristics built from positive Fourier coefficients.  In the pure finite-pole model, the correct determinant characteristic factors into a polynomial whose zeros are the reciprocals of the exterior poles.  In the presence of a holomorphic background, finite sampling, and noise, roots of individual determinants are unstable and are used only as local evidence.  We aggregate this evidence into a scalar indicator field on the reciprocal pole plane: at each sampling point, the field records the fraction of determinant orders and shifts for which a small contour centered at that point encloses exactly one empirical determinant zero.  The resulting field plays the role of a sampling-type imaging functional for pole visibility.  Fixed superlevel sets give visible-pole clusters, while zero-dimensional persistent homology is used only as a threshold-robust post-processing step.  We prove deterministic results linking pure-pole contour counts, Rouch\'e stability, indicator-field contrast, fixed-threshold component recovery, and persistence-gap stability.  These results explain why isolated poles with sufficient residue and separation generate stable high-value components, whereas weak, close, boundary-near, or noise-dominated poles may give low, short-lived, or merged components.  The framework is a cluster-certification and imaging method, not an unconditional all-pole recovery procedure.}


\maketitle

\section{Introduction}
\label{sec:introduction}
Numerical analytic continuation is a classical problem in mathematical
analysis, with important applications in scientific computing and engineering
\cite{Lavrentiev1967}.  Its main difficulty is severe ill-conditioning:
small perturbations in the input data may be strongly amplified after
continuation, so direct numerical continuation is unstable without additional
a priori information \cite{Cannon1965}.  Substantial work has therefore been
devoted to both the theoretical and computational stabilization of analytic
continuation, including bounded continuation, conformal-mapping approaches,
regularization methods, rational approximation, and fast Fourier based
procedures
\cite{Bisshopp1983,Fu2009,Henrici1966,Miller1970,Reichel1986,
Stefanescu1980,Trefethen2020,Trefethen2023}.  Much of this literature concerns
the stable continuation or approximation of holomorphic functions, or the
reconstruction of analytic objects under prescribed regularity, boundedness, or
rational-approximation assumptions.  In contrast, the present paper focuses on
a pole-oriented form of meromorphic continuation: rather than attempting direct
pointwise extrapolation, we seek stable information about exterior poles that
obstruct holomorphic continuation.

The pole-recovery viewpoint goes back at least to Miller's stabilized numerical
analytic prolongation with poles \cite{Miller1970}.  Miller formulated
meromorphic continuation as a stabilized inverse problem and introduced the
SNAPP algorithm, which seeks a rational approximant with the minimum number of
poles compatible with the data and a prescribed global bound.  This work
already identifies pole recovery as a natural stabilized output of meromorphic
continuation.  More recently, Derevianko \cite{Derevianko2025} developed a
different Fourier-domain approach for recovering rational functions from their
Fourier coefficients.  By exploiting the finite exponential-sum structure of
these coefficients, the method reconstructs poles inside and outside the unit
circle through special Hankel matrix pencils and gives a sensitivity analysis
for the recovered pole locations under structured and unstructured
perturbations.
The same algebraic backbone appears in Prony-type methods, ESPRIT, matrix
pencils, and Hankel structured techniques for estimating parameters of
exponential sums \cite{Potts2013,Sarkar1995}.  Recent developments have further
extended these ideas to rational approximation, exponential-sum reconstruction,
noisy analytic continuation, and inverse problems
\cite{Deng2026_a,Deng2026_b,Deng2026_c,DereviankoHuebner2025,
DereviankoPlonkaPetz2021,DereviankoPlonka2022,Ying2022,Ying2022_b}.  For
example, Ying \cite{Ying2022,Ying2022_b} used conformal or M\"obius
transformations together with Fourier-domain Prony-type recovery to extract
pole or spectral information from limited noisy data.  These works show that
Fourier or frequency-domain data, finite exponential-sum structure, and
Hankel--Prony mechanisms provide a powerful framework for pole and parameter
recovery.
Most directly related to the present work is the determinant-certification
approach in \cite{YangDeng2026}, which studies outward
meromorphic continuation from circular boundary data.  In that work, shifted
determinant characteristics are built from positive Fourier coefficients of
the boundary trace.  In the pure finite-pole model, the correct-order
determinant factors into a nonzero constant times the polynomial whose zeros
are the reciprocal exterior poles.  For data contaminated by holomorphic
background terms, discretization, and noise, roots of empirical determinants
are treated only as candidates; local argument-principle counts, contour
moments, empirical margins, and persistence over determinant orders and shifts
are then used to certify visible reciprocal poles.  Thus, that work follows a
``root-propose and contour-certify'' principle: determinant roots suggest
candidate pole regions, while local contour counts provide certification.

The present paper takes a different viewpoint.  We do not use Hankel ranks,
Prony roots, or certified determinant roots as final pole decisions.  Instead,
we construct a contour-count indicator field on the reciprocal pole plane.  The
starting point is still a family of determinant characteristics built from
positive Fourier coefficients of the boundary trace.  However, rather than
selecting individual candidate roots and certifying them one by one, we
evaluate local one-zero contour decisions over a moving family of sampling
points.  At each sampling point \(\lambda\), these binary decisions are
aggregated over determinant orders and shifts to define an indicator value
\(P_\rho(\lambda)\).  Large values of \(P_\rho\) mark regions where the
evidence for a visible reciprocal pole is stable across the testing family.

This field plays the role of an imaging functional.  The analogy is with
sampling-type methods in inverse scattering, where an indicator function is
evaluated over a search domain and its high-value regions reveal scatterers or
point-like targets, e.g., see \cite{CakoniColtonMonk2011, Ito2012, LiZou2013, LiuSun2018}.  Here the indicator is not based on a far-field operator or
a range test; it is based on argument-principle zero counts of determinant
characteristics.  Thus the method converts meromorphic pole detection into a
visible-pole imaging problem in the reciprocal plane.  Fixed superlevel sets of
\(P_\rho\) already provide candidate visible pole clusters, while persistent
homology may be used as a post-processing tool to select components that are
robust under changes of the threshold.

The paper is designed as a numerical-methods contribution in the style of scientific computing: the main object is an implementable imaging functional, its stability mechanisms are made explicit, and the topology layer is used as a robustness diagnostic rather than as a substitute for the indicator itself.  The contributions are as follows.
\begin{enumerate}[label=(\roman*)]
\item We formulate a contour-count indicator field for meromorphic pole visibility from Fourier determinant characteristics.
\item We prove that, in the pure finite-pole model, the ideal indicator is supported near reciprocal poles in the precise sense of one-zero contour counts.
\item We give a Rouch\'e-based stability mechanism showing when the empirical indicator retains high values near visible poles under coefficient perturbations.
\item We prove fixed-threshold and persistence-gap stability statements for the connected components of the indicator field.
\item We present an algorithmic framework in which fixed thresholding gives visible clusters and persistent homology is used as an optional threshold-robust post-processing step.
\end{enumerate}

The paper is organized as follows.  Section~\ref{sec:local-evidence} recalls the meromorphic continuation model and determinant-count data.  Section~\ref{sec:indicator-field} defines the contour-count indicator field and proves its pure-pole and Rouch\'e stability properties.  Section~\ref{sec:cluster-extraction} studies fixed-threshold extraction of visible-pole clusters.  And  the persistent post-processing layer is developed.    Section~\ref{sec:numerics} describes the numerical experiments.   Section~\ref{sec:conclusions} concludes the paper.

\section{From determinant counts to local pole evidence}
\label{sec:local-evidence}

This section fixes the notation needed for the indicator construction.  The
determinant identities and contour-count certification mechanism are not
reproved here; they are used as a local source of pole evidence.  One can refer to \cite{YangDeng2026} for more details.  The purpose
of the present paper is to turn these local one-zero decisions into an imaging
field on the reciprocal pole plane.

Let $\D=\{z\in\C:|z|<1\}$ denote the open unit disk, and  $\D_R=\{z\in\C:|z|<R\}$ be its extension for som $R>1$. 
For any domain \(\Omega\subset\C\), the notations \(\Ocal(\Omega)\) and
\(\Mcal(\Omega)\) stand for the spaces of holomorphic and meromorphic functions in
\(\Omega\), respectively. 
We  investigate functions \(f\in\Ocal(\D)\) that possess an outward meromorphic continuation
to \(F\in\Mcal(\D_R)\) characterized by
\begin{equation}
    F(z)=h(z)+\sum_{j=1}^{N}\frac{r_j}{z-p_j}, \qquad h\in\Ocal(\D_{\rho_h}), \qquad  1<|p_j|<R<\rho_h .
    \label{eq:model}
\end{equation}
The exterior poles \(\{p_j\}_{j=1}^N\) are assumed to be simple and distinct, and located within the physical annulus $1<|p_j|<R$, with non-vanishing residues \(r_j\neq0\).

To reformulate the topological search into an algebraic problem, we introduce the reciprocal variables
%The physical pole annulus \(1<|p|<R\) is represented in reciprocal variables by
\begin{equation}
    \lambda_j=\frac1{p_j},
    \qquad
    U_R=\left\{\lambda\in\C:\frac1R<|\lambda|<1\right\}.
    \label{eq:reciprocal-region}
\end{equation}
Consequently, identifying the visible exterior poles in the physical domain is equivalent to locating their reciprocal images $\lambda_j$ within the open annulus $U_R$.
%Thus visible exterior poles are searched for as visible reciprocal features in
%the annulus \(U_R\).

Let the Taylor expansion of $f(z)$ inside the unit disk be given by
\[
    f(z)=\sum_{k=0}^{\infty}a_k z^k,
    \qquad |z|<1,
\]
whose coefficients $a_k$ exactly coincide with the positive Fourier coefficients of the boundary trace on $\partial\Ocal(\D)$.
 Under the structure assumption \eqref{eq:model}, these spectral coefficients admit the decomposition
\begin{equation}
    a_k=h_k+\sum_{j=1}^{N}c_j\lambda_j^k,
    \qquad
    c_j=-\frac{r_j}{p_j},
    \qquad k=0,1,2,\ldots,
    \label{eq:coefficient-model}
\end{equation}
where $h_k$ are the Taylor coefficients of $h(z)$. It is worth noting that \eqref{eq:coefficient-model} precisely embeds the pole contributions into a finite exponential sum parameterized by the reciprocal variables \(\lambda _{j}\), superposed on the holomorphic background sequence \((h_k)\).
%Hence the pole contribution is a finite exponential sum in the reciprocal
%variables, while the holomorphic background contributes the additional sequence
%\((h_k)\).

In computation, the boundary data are sampled at \(M\) equispaced points and may
be contaminated by noise.  We denote the empirical Fourier coefficients by
\(a_k^{\delta,M}\).  For a trial order \(n\geq1\) and shift \(L\geq0\), define
the empirical determinant characteristic
\begin{equation}
\mathcal D_{n,L}^{\delta,M}(\lambda)
=
\det
\begin{pmatrix}
a_L^{\delta,M} & a_{L+1}^{\delta,M} & \cdots & a_{L+n}^{\delta,M}\\
a_{L+1}^{\delta,M} & a_{L+2}^{\delta,M} & \cdots & a_{L+n+1}^{\delta,M}\\
\vdots & \vdots & & \vdots\\
a_{L+n-1}^{\delta,M} & a_{L+n}^{\delta,M} & \cdots & a_{L+2n-1}^{\delta,M}\\
1 & \lambda & \cdots & \lambda^n
\end{pmatrix}.
\label{eq:empirical-det}
\end{equation}
The coefficient availability condition is
\[
    L+2n\leq M.
\]

In the ideal finite-pole case, and for the correct order, the zeros of the
corresponding determinant characteristic coincide with the reciprocal poles.
With background, sampling, and noise, a zero of a single determinant is no
longer accepted as a pole.  It is only local evidence.  To test this evidence,
we use argument-principle counts on small contours.

Let \(\mathcal I\) be a finite family of determinant orders and shifts,
\begin{equation}
    \mathcal I = \{(n,L): n_{\min}\leq n\leq n_{\max},\ L\in\mathcal L\},
    \label{eq:testing-family}
\end{equation}
where \[
    \mathcal L=\{L_1,\ldots,L_s\}\subset\mathbb N_0
\]
is a finite set of shifts and \(n_{\min}\) and \(n_{\max}\) are the minimum
and maximum trial orders.
For a point \(\lambda\in U_R\) and a radius \(\rho>0\), set
\[
    \Gamma_\rho(\lambda) =  \{\xi\in\C:|\xi-\lambda|=\rho\}.
\]
We restrict attention to the shrunk search domain
\begin{equation}
    U_R^\rho =  \{\lambda\in U_R:\dist(\lambda,\partial U_R)>\rho\},
    \label{eq:shrunk-domain}
\end{equation}
so that \(\Gamma_\rho(\lambda)\subset U_R\).
For each \((n,L)\in\mathcal I\), define the local contour count
\begin{equation}
  C_{n,L}^{\delta,M}(\Gamma_\rho(\lambda)) =
    \frac{1}{2\pi i} \int_{\Gamma_\rho(\lambda)} \frac{(\mathcal D_{n,L}^{\delta,M})'(\xi)}{\mathcal D_{n,L}^{\delta,M}(\xi)} \,\dd\xi ,
    \label{eq:local-count}
\end{equation}
whenever the determinant has no zero on the contour.  By the argument
principle, this integer counts the zeros of
\(\mathcal D_{n,L}^{\delta,M}\) enclosed by the local circle.
The basic local evidence used in this paper is the one-zero decision
\begin{equation}
    E_{n,L}^{\rho}(\lambda)
    =
    \mathbf 1_{\{
    C_{n,L}^{\delta,M}(\Gamma_\rho(\lambda))=1
    \}}.
    \label{eq:local-evidence}
\end{equation}
Thus \(E_{n,L}^{\rho}(\lambda)=1\) means that, for the determinant
characteristic indexed by \((n,L)\), the circle centered at \(\lambda\) encloses
exactly one empirical determinant zero.  It is a local indicator of one-pole
evidence at scale \(\rho\).

The central idea of the present paper is to aggregate these binary local
decisions over the testing family \(\mathcal I\).  The resulting averaged field
is not a root set and not a Hankel-rank estimator.  It is a contour-count
indicator field on the reciprocal search annulus.  This field is introduced in
the next section.

\section{Contour-count indicator fields}
\label{sec:indicator-field}

As explained in the  previous section, the one-zero decisions \(E_{n,L}^{\rho}(\lambda)\) are local and binary.  To
obtain an imaging functional, we average them over the testing family.

%Fix a contour radius \(\rho>0\).  For \(\lambda\in\C\), write
%\[
%    \Gamma_\rho(\lambda)
%    =
%    \{\xi\in\C:|\xi-\lambda|=\rho\}.
%\]
%To keep the testing contour inside the reciprocal search annulus, define the
%shrunk search region
%\begin{equation}
%    U_R^\rho
%    =
%    \{\lambda\in U_R:\dist(\lambda,\partial U_R)>\rho\}.
%    \label{eq:shrunk-domain}
%\end{equation}

\begin{definition}[Contour-count indicator field]
\label{def:indicator-field}
For \(\lambda\in U_R^\rho\), define
\begin{equation}
    P_\rho^{\delta,M}(\lambda)
    =
    \frac1{|\Ical|}
    \#
    \left\{
    (n,L)\in\Ical:
    C_{n,L}^{\delta,M}(\Gamma_\rho(\lambda))=1
    \right\}.
    \label{eq:indicator-field}
\end{equation}
\end{definition}

The field satisfies \(0\leq P_\rho^{\delta,M}(\lambda)\leq1\).  A large value
means that many determinant characteristics agree that the local circle centered
at \(\lambda\) encloses exactly one empirical determinant zero.  Thus
\(P_\rho^{\delta,M}\) is a soft indicator of visible pole evidence, not a
characteristic function in the strict \(0\)-\(1\) sense.

For each \((n,L)\in\Ical\), define
\[
    m_{n,L}^{\delta,M}(\lambda,\rho)
    =
    \min_{\xi\in\Gamma_\rho(\lambda)}
    |\mathcal D_{n,L}^{\delta,M}(\xi)| .
\]
The empirical contour-margin field is the median of these values:
\begin{equation}
    \mathfrak m_\rho^{\delta,M}(\lambda)
    =
    \operatorname{med}
    \left\{
    m_{n,L}^{\delta,M}(\lambda,\rho):(n,L)\in\Ical
    \right\}.
    \label{eq:margin-field}
\end{equation}
Here, for a finite real set \(\{x_1,\ldots,x_q\}\), \(\operatorname{med}\) is
the middle value after nondecreasing rearrangement if \(q\) is odd, and the
average of the two middle values if \(q\) is even.
%We also use the empirical contour-margin field.  For each
%\((n,L)\in\Ical\), set
%\[
%    m_{n,L}^{\delta,M}(\lambda,\rho)
%    =
%    \min_{\xi\in\Gamma_\rho(\lambda)}
%    |\mathcal D_{n,L}^{\delta,M}(\xi)| .
%\]
%Let
%\[
%    \{m_{(1)}^{\delta,M}(\lambda,\rho),\ldots,
%    m_{(q)}^{\delta,M}(\lambda,\rho)\},
%    \qquad q=|\Ical|,
%\]
%be the nondecreasing rearrangement of the values
%\[
%    \{m_{n,L}^{\delta,M}(\lambda,\rho):(n,L)\in\Ical\}.
%\]
%We define
%\begin{equation}
%    \mathfrak m_\rho^{\delta,M}(\lambda)
%    =
%    \begin{cases}
%    m_{((q+1)/2)}^{\delta,M}(\lambda,\rho),
%    & q \text{ odd},\\[4pt]
%    \dfrac12\left(
%    m_{(q/2)}^{\delta,M}(\lambda,\rho)
%    +
%    m_{(q/2+1)}^{\delta,M}(\lambda,\rho)
%    \right),
%    & q \text{ even}.
%    \end{cases}
%    \label{eq:margin-field}
%\end{equation}
Small margin values indicate that a contour passes close to a determinant zero,
where winding-number computations and moment estimates may be unstable.  A
margin-filtered indicator can be defined by
\begin{equation}
    P_{\rho,\tau}^{\delta,M}(\lambda)
    =
    P_\rho^{\delta,M}(\lambda)
    \mathbf 1_{\{\mathfrak m_\rho^{\delta,M}(\lambda)\geq\tau_{\rm cont}\}}.
    \label{eq:margin-filtered-indicator}
\end{equation}
Unless otherwise stated, \(P_\rho\) denotes the indicator field used in the
subsequent analysis, either the raw field \(P_\rho^{\delta,M}\) or the
margin-filtered field \(P_{\rho,\tau}^{\delta,M}\).

\subsection{Ideal contour-count geometry}
\label{subsec:ideal-geometry}

The indicator interpretation is clearest in the exact pure-pole case.  Suppose
that \(h\equiv0\) and that no sampling or measurement error is present.  For the
correct determinant order \(n=N\), the pure-pole determinant characteristic has
the factorization
\[
    \mathcal D_{N,L}^{\rm p}(\lambda)
    =
    A_L\prod_{j=1}^{N}(\lambda-\lambda_j),
    \qquad A_L\neq0.
\]
Thus the determinant zeros are exactly the reciprocal poles
\(\lambda_1,\ldots,\lambda_N\).

\begin{proposition}[One-zero geometry of the ideal indicator]
\label{prop:ideal-one-zero}
Assume the pure-pole model and the correct determinant order \(n=N\).  Let
\(\Gamma_\rho(\lambda)\) avoid all reciprocal poles.  Then
\begin{equation}
    \frac1{2\pi i}
    \int_{\Gamma_\rho(\lambda)}
    \frac{(\mathcal D_{N,L}^{\rm p})'(\xi)}
         {\mathcal D_{N,L}^{\rm p}(\xi)}
    \,\dd\xi
    =
    \#\{j:|\lambda-\lambda_j|<\rho\}.
    \label{eq:ideal-count-geometry}
\end{equation}
Consequently, the local contour gives a one-zero response exactly when the disk
\(B(\lambda,\rho)\) contains one reciprocal pole.
\end{proposition}

\begin{proof}
The zeros of \(\mathcal D_{N,L}^{\rm p}\) are exactly
\(\lambda_1,\ldots,\lambda_N\), each with multiplicity one.  The argument
principle gives the number of these zeros inside \(\Gamma_\rho(\lambda)\),
which is exactly \(\#\{j:|\lambda-\lambda_j|<\rho\}\).
\end{proof}

This proposition explains why \(P_\rho\) is naturally an imaging functional.
The high-response set is not a point set.  Even in the ideal case, a pole at
\(\lambda_j\) produces a response over all contour centers \(\lambda\) satisfying
\[
    |\lambda-\lambda_j|<\rho.
\]
The spatial width of the response is therefore controlled by the contour radius
\(\rho\), together with the spread of empirical determinant zeros across the
testing family.

\begin{corollary}[Separated ideal response]
\label{cor:separated-ideal-response}
Assume the pure-pole model and the correct determinant order \(n=N\).  Suppose
\[
    0<\rho<\frac12\min_{i\neq j}|\lambda_i-\lambda_j|.
\]
Then the disks \(B(\lambda_j,\rho)\) are disjoint.  For the correct-order
determinant, the one-zero response set is
\begin{equation}
    \bigcup_{j=1}^{N}B(\lambda_j,\rho),
    \label{eq:ideal-response-union}
\end{equation}
up to boundary circles where the contour passes through a reciprocal pole.
\end{corollary}

\begin{proof}
The separation condition makes the disks \(B(\lambda_j,\rho)\) pairwise
disjoint.  By Proposition~\ref{prop:ideal-one-zero}, a center \(\lambda\) gives
count one if and only if \(B(\lambda,\rho)\) contains exactly one reciprocal
pole.  This is precisely the union in \eqref{eq:ideal-response-union}, excluding
the boundary cases \(|\lambda-\lambda_j|=\rho\).
\end{proof}

If the testing family contains several shifts with the correct order, the same
geometry holds for each such determinant.  Thus, in the ideal separated case,
averaging over these correct-order tests produces high indicator values on the
union of the disks \(B(\lambda_j,\rho)\).  Additional trial orders may contribute
extra empirical evidence, but the clean geometric interpretation above is tied
to the correct-order pure-pole determinant.

\subsection{Rouch\'e stability of the indicator response}
\label{subsec:rouche-indicator}

The ideal response persists under perturbations when the determinant
perturbation is smaller than the corresponding contour margin.  We first state
this for a correct-order pure-pole reference determinant.  For a fixed shift
\(L\), define the reference margin on \(\Gamma_\rho(\lambda)\) by
\begin{equation}
    m_{N,L}^{\rm p}(\lambda,\rho)
    =
    \min_{\xi\in\Gamma_\rho(\lambda)}
    |\mathcal D_{N,L}^{\rm p}(\xi)|.
    \label{eq:reference-margin}
\end{equation}
Let \(\widetilde{\mathcal D}_{N,L}\) be a perturbed determinant and set
\begin{equation}
    \Delta_{N,L}(\lambda,\rho)
    =
    \max_{\xi\in\Gamma_\rho(\lambda)}
    |\widetilde{\mathcal D}_{N,L}(\xi)-\mathcal D_{N,L}^{\rm p}(\xi)|.
    \label{eq:det-perturb-on-circle}
\end{equation}

\begin{theorem}[Rouch\'e stability of local indicator counts]
\label{thm:rouche-indicator}
Assume that the pure-pole reference determinant
\(\mathcal D_{N,L}^{\rm p}\) has no zero on \(\Gamma_\rho(\lambda)\).  If
\begin{equation}
    \Delta_{N,L}(\lambda,\rho)<m_{N,L}^{\rm p}(\lambda,\rho),
    \label{eq:rouche-indicator-condition}
\end{equation}
then \(\widetilde{\mathcal D}_{N,L}\) and
\(\mathcal D_{N,L}^{\rm p}\) have the same number of zeros inside
\(\Gamma_\rho(\lambda)\), counted with multiplicity.  In particular, if
\(B(\lambda,\rho)\) contains exactly one reciprocal pole of the reference
model, then
\[
    \frac1{2\pi i}
    \int_{\Gamma_\rho(\lambda)}
    \frac{\widetilde{\mathcal D}_{N,L}'(\xi)}
         {\widetilde{\mathcal D}_{N,L}(\xi)}
    \,\dd\xi
    =1.
\]
\end{theorem}

\begin{proof}
Condition \eqref{eq:rouche-indicator-condition} is the hypothesis of
Rouch\'e's theorem on \(\Gamma_\rho(\lambda)\), applied to
\(\mathcal D_{N,L}^{\rm p}\) and
\(\widetilde{\mathcal D}_{N,L}-\mathcal D_{N,L}^{\rm p}\).  Therefore
\(\widetilde{\mathcal D}_{N,L}\) and \(\mathcal D_{N,L}^{\rm p}\) have the same
number of zeros inside the contour.  The final statement follows from
Proposition~\ref{prop:ideal-one-zero} and the argument principle.
\end{proof}

The following consequence explains how stable correct-order tests contribute to
the indicator field.

\begin{corollary}[Indicator lower bound from stable correct-order tests]
\label{cor:field-lower-bound}
Fix \(\lambda\in U_R^\rho\).  Let
\[
    \Ical_N(\lambda)
    \subseteq
    \{(N,L)\in\Ical\}
\]
be the set of correct-order tests for which \(B(\lambda,\rho)\) contains exactly
one reciprocal pole of the reference model and the Rouch\'e condition
\eqref{eq:rouche-indicator-condition} holds.  Then
\begin{equation}
    P_\rho^{\delta,M}(\lambda)
    \geq
    \frac{|\Ical_N(\lambda)|}{|\Ical|}.
    \label{eq:field-lower-bound}
\end{equation}
\end{corollary}

\begin{proof}
For every \((N,L)\in\Ical_N(\lambda)\), Theorem~\ref{thm:rouche-indicator}
gives
\[
    C_{N,L}^{\delta,M}(\Gamma_\rho(\lambda))=1.
\]
Each such pair therefore contributes one hit to the numerator in
\eqref{eq:indicator-field}.  Dividing by \(|\Ical|\) gives
\eqref{eq:field-lower-bound}.
\end{proof}

More generally, trial orders different from \(N\) may also produce stable
one-zero responses near a visible pole.  These responses are included
empirically in \(P_\rho^{\delta,M}\), but the exact geometric interpretation in
Proposition~\ref{prop:ideal-one-zero} is reserved for the correct-order
pure-pole determinant.  This distinction is important: the indicator field is
not a rank estimator and does not require every contributing determinant to have
an exact pole-factorization formula.  It only aggregates stable local one-zero
evidence over the testing family.

The visibility mechanism is therefore simple.  A pole is visible when a
non-negligible fraction of determinant tests give stable one-zero evidence
around it.  Small residues, close reciprocal poles, boundary-near locations,
large shifts, and noise all reduce contour margins or increase determinant
perturbations, thereby lowering the indicator value.

\section{Visible-cluster extraction from the indicator field}
\label{sec:cluster-extraction}

The contour-count indicator field can be used directly, in the same way that
sampling-type imaging functionals are used in inverse scattering.  A fixed
superlevel threshold gives visible-pole clusters, while persistent homology may
be used as a post-processing tool to reduce dependence on the chosen threshold.
This section describes both extraction mechanisms and the resulting algorithm.

\subsection{Fixed-threshold visible-pole clusters}
\label{subsec:fixed-threshold}

For a prescribed threshold \(\alpha\in(0,1)\), define the superlevel set
\begin{equation}
    \Omega_{\rho,\alpha}
    =
    \{\lambda\in U_R^\rho:P_\rho(\lambda)\geq\alpha\}.
    \label{eq:fixed-superlevel}
\end{equation}
The connected components of \(\Omega_{\rho,\alpha}\) are called
fixed-threshold visible-pole clusters.  If \(\Ccal\) is such a component, a
representative reciprocal pole location may be estimated by the weighted center
\begin{equation}
    \widehat\lambda(\Ccal)
    =
    \frac{\int_{\Ccal}\lambda P_\rho(\lambda)\,\dd A(\lambda)}
         {\int_{\Ccal}P_\rho(\lambda)\,\dd A(\lambda)},
    \label{eq:continuous-center}
\end{equation}
provided the denominator is nonzero.  On a grid \(\Lambda_h\), this becomes
\begin{equation}
    \widehat\lambda_h(\Ccal)
    =
    \frac{\sum_{\lambda\in \Ccal_h}\lambda P_\rho(\lambda)}
         {\sum_{\lambda\in \Ccal_h}P_\rho(\lambda)},
    \qquad
    \Ccal_h=\Ccal\cap\Lambda_h .
    \label{eq:discrete-center}
\end{equation}
The corresponding physical pole estimate is
\[
    \widehat p(\Ccal)=\frac1{\widehat\lambda(\Ccal)} .
\]

The following deterministic result explains when fixed-threshold components are
stable under perturbations of the indicator field.

\begin{theorem}[Fixed-threshold component stability]
\label{thm:threshold-stability}
Let \(P,\widetilde P:U\to[0,1]\) satisfy
\begin{equation}
    \|P-\widetilde P\|_{L^\infty(U)}\leq\varepsilon .
    \label{eq:field-uniform-bound}
\end{equation}
Suppose there are compact sets \(Q\subset V\subset U\) and levels
\(\alpha_{\rm in}>\alpha_{\rm out}\) such that
\begin{equation}
    P(\lambda)\geq\alpha_{\rm in}\quad\text{on }Q,
    \qquad
    P(\lambda)\leq\alpha_{\rm out}\quad\text{on }\partial V,
    \label{eq:threshold-separation}
\end{equation}
and \(Q\) lies in a connected component of
\(\Omega_{\alpha_{\rm in}}(P)\) contained in \(V\).  If
\begin{equation}
    \alpha_{\rm out}+\varepsilon<\alpha<\alpha_{\rm in}-\varepsilon,
    \label{eq:admissible-alpha}
\end{equation}
then every connected component of \(\Omega_\alpha(\widetilde P)\) intersecting
\(Q\) is contained in \(V\).  In particular, such a component cannot connect to
the exterior of \(V\) at level \(\alpha\).
\end{theorem}

\begin{proof}
For \(\lambda\in Q\), \eqref{eq:field-uniform-bound} and
\eqref{eq:threshold-separation} give
\[
    \widetilde P(\lambda)
    \geq
    P(\lambda)-\varepsilon
    \geq
    \alpha_{\rm in}-\varepsilon
    >
    \alpha .
\]
Hence \(Q\subset\Omega_\alpha(\widetilde P)\).  On the other hand, for
\(\lambda\in\partial V\),
\[
    \widetilde P(\lambda)
    \leq
    P(\lambda)+\varepsilon
    \leq
    \alpha_{\rm out}+\varepsilon
    <
    \alpha .
\]
Thus
\[
    \partial V\cap\Omega_\alpha(\widetilde P)=\emptyset .
\]
Any connected component of \(\Omega_\alpha(\widetilde P)\) that intersects
\(Q\) must therefore remain inside \(V\), since reaching the exterior of \(V\)
would require crossing \(\partial V\).  This proves the claim.
\end{proof}

The theorem does not prescribe a universal threshold.  It states that any
threshold lying between the interior response and the boundary background, with
a margin exceeding the field perturbation, recovers the same visible component.
In computations, this is the regime where the high-value components of
\(P_\rho\) are well separated from low-level spurious responses.

\subsection{Persistent post-processing}
\label{subsec:persistent-postprocessing}

Fixed thresholding is direct and often sufficient.  However, the choice of
\(\alpha\) may depend on noise, contour radius, pole separation, and the testing
family.  Persistent homology is used here only as a post-processing technique to
reduce threshold dependence.  It does not replace the indicator field.

For \(\beta\in[0,1]\), define the superlevel set
\begin{equation}
    \Omega_\beta(P_\rho)
    =
    \{\lambda\in U_R^\rho:P_\rho(\lambda)\geq\beta\}.
    \label{eq:superlevel-filtration}
\end{equation}
As \(\beta\) decreases, these sets form a nested superlevel filtration.  The
connected components of the sets \(\Omega_\beta(P_\rho)\) define the
zero-dimensional persistence of \(P_\rho\).  A component born at level \(b\) and
merged at level \(d\) has lifetime
\[
    \ell=b-d.
\]
Long lifetimes indicate high-value components that persist across a range of
thresholds.

\begin{lemma}[Superlevel-set inclusion]
\label{lem:superlevel-inclusion}
Let \(P,\widetilde P:U\to[0,1]\) satisfy
\[
    \|P-\widetilde P\|_{L^\infty(U)}\leq\varepsilon .
\]
Then, for every \(\beta\in[\varepsilon,1-\varepsilon]\),
\begin{equation}
    \Omega_{\beta+\varepsilon}(P)
    \subseteq
    \Omega_\beta(\widetilde P)
    \subseteq
    \Omega_{\beta-\varepsilon}(P).
    \label{eq:superlevel-inclusion}
\end{equation}
\end{lemma}

\begin{proof}
If \(P(\lambda)\geq\beta+\varepsilon\), then
\[
    \widetilde P(\lambda)
    \geq
    P(\lambda)-\varepsilon
    \geq
    \beta ,
\]
which proves the first inclusion.  If \(\widetilde P(\lambda)\geq\beta\), then
\[
    P(\lambda)
    \geq
    \widetilde P(\lambda)-\varepsilon
    \geq
    \beta-\varepsilon ,
\]
which proves the second inclusion.
\end{proof}

\begin{theorem}[Persistence-gap stability]
\label{thm:persistence-gap}
Let \(P,\widetilde P:U\to[0,1]\) satisfy
\[
    \|P-\widetilde P\|_{L^\infty(U)}\leq\varepsilon .
\]
Suppose there exist compact sets \(Q\subset V\subset U\) and levels
\(\alpha_{\rm in}>\alpha_{\rm out}\) such that
\[
    P(\lambda)\geq\alpha_{\rm in}\quad\text{on }Q,
    \qquad
    P(\lambda)\leq\alpha_{\rm out}\quad\text{on }\partial V,
\]
and \(Q\) lies in a connected component of
\(\Omega_{\alpha_{\rm in}}(P)\) contained in \(V\).  If
\[
    \alpha_{\rm in}-\alpha_{\rm out}>2\varepsilon,
\]
then every component of
\(\Omega_{\alpha_{\rm in}-\varepsilon}(\widetilde P)\) intersecting \(Q\) is
contained in \(V\) and cannot merge with the exterior of \(V\) before level
\(\alpha_{\rm out}+\varepsilon\).  Hence the corresponding perturbed component
has lifetime at least
\begin{equation}
    \alpha_{\rm in}-\alpha_{\rm out}-2\varepsilon .
    \label{eq:lifetime-lower-bound}
\end{equation}
\end{theorem}

\begin{proof}
By Lemma~\ref{lem:superlevel-inclusion},
\[
    Q\subseteq\Omega_{\alpha_{\rm in}-\varepsilon}(\widetilde P).
\]
For \(\lambda\in\partial V\), one has
\[
    \widetilde P(\lambda)
    \leq
    P(\lambda)+\varepsilon
    \leq
    \alpha_{\rm out}+\varepsilon .
\]
Therefore, at every level larger than \(\alpha_{\rm out}+\varepsilon\), no
component of the superlevel set of \(\widetilde P\) that intersects \(Q\) can
cross \(\partial V\).  Such a component remains isolated from the exterior over
the level interval
\[
    (\alpha_{\rm out}+\varepsilon,\alpha_{\rm in}-\varepsilon].
\]
The length of this interval is
\[
    \alpha_{\rm in}-\alpha_{\rm out}-2\varepsilon ,
\]
which gives the stated lower bound.
\end{proof}

\begin{definition}[Persistent visible pole component]
\label{def:persistent-component}
A connected component \(\Ccal\) of the superlevel filtration of \(P_\rho\) is
called a persistent visible pole component if its lifetime satisfies
\[
    \ell(\Ccal)\geq\tau_{\rm life},
\]
its empirical contour margin satisfies
\[
    \operatorname{med}_{\lambda\in\Ccal_h}
    \mathfrak m_\rho^{\delta,M}(\lambda)
    \geq\tau_{\rm cont},
\]
and it is spatially disjoint from previously accepted components.
\end{definition}

The estimated number of persistent visible pole components is
\begin{equation}
    \widehat N_{\rm pers}
    =
    \#\{\Ccal:\Ccal\text{ is an accepted persistent visible pole component}\}.
    \label{eq:Npers}
\end{equation}
This number is a cluster count.  A component may represent one visible pole, or
a cluster of close poles that cannot be separated at the chosen contour radius
and noise level.

\subsection{Algorithmic implementation}
\label{subsec:algorithmic-implementation}

Algorithm~\ref{alg:indicator-field} summarizes the complete procedure.  The
fixed-threshold step is the primary extraction method; persistent homology is an
optional post-processing step for threshold robustness.

\begin{algorithm}
\caption{Contour-count indicator field for visible pole clusters}
\label{alg:indicator-field}
\begin{algorithmic}[1]
\Require Empirical coefficients \(a_k^{\delta,M}\), search radius \(R\),
testing family \(\Ical\), contour radius \(\rho\), grid
\(\Lambda_h\subset U_R^\rho\), contour-margin threshold \(\tau_{\rm cont}\),
superlevel threshold \(\alpha\), optional persistence threshold
\(\tau_{\rm life}\)
\For{\((n,L)\in\Ical\)}
    \State Construct \(\mathcal D_{n,L}^{\delta,M}\)
\EndFor
\For{\(\lambda\in\Lambda_h\)}
    \State Set \(\Gamma_\rho(\lambda)=\{\xi:|\xi-\lambda|=\rho\}\)
    \State Compute \(C_{n,L}^{\delta,M}(\Gamma_\rho(\lambda))\) for all
    \((n,L)\in\Ical\)
    \State Compute \(P_\rho^{\delta,M}(\lambda)\) by
    \eqref{eq:indicator-field}
    \State Compute \(\mathfrak m_\rho^{\delta,M}(\lambda)\) by
    \eqref{eq:margin-field}
\EndFor
\State Apply the margin filter \eqref{eq:margin-filtered-indicator}, if desired
\State Extract connected components of
\(\{\lambda\in\Lambda_h:P_\rho(\lambda)\geq\alpha\}\)
\State Estimate component centers by \eqref{eq:discrete-center}
\State Optionally compute \(H_0\) persistence of the superlevel filtration and
retain components satisfying Definition~\ref{def:persistent-component}
\State Return visible component centers \(\widehat\lambda(\Ccal)\) and physical
pole estimates \(\widehat p(\Ccal)=1/\widehat\lambda(\Ccal)\)
\end{algorithmic}
\end{algorithm}

The parameters have separate roles.  The contour radius \(\rho\) controls
spatial resolution.  Small \(\rho\) gives sharper but potentially fragmented
responses, while large \(\rho\) smooths the field but may merge close poles.
The threshold \(\alpha\) selects high-response visible clusters.  The margin
threshold \(\tau_{\rm cont}\) removes unstable contours.  The persistence
threshold \(\tau_{\rm life}\), when used, removes components that are
short-lived across threshold levels.

\begin{remark}[Scope and limitations]
\label{rem:scope-limitations}
The method is a visible-cluster imaging and certification procedure.  It does
not guarantee recovery of every formal pole in the meromorphic model.  The
output is a set of fixed-threshold or persistent visible components in the
reciprocal plane.  A component may correspond to one pole, to several close
poles that are not separated at the chosen radius, or to a feature that is
visible only over a limited threshold range.

The main mechanisms that reduce visibility are small residues, large physical
pole modulus, proximity to the boundary of \(U_R\), close reciprocal poles,
holomorphic background terms comparable to the pole contribution, and noise
that changes local winding counts.  These are not limitations of the
topological post-processing layer alone; they reflect the underlying stability
balance between determinant perturbation and contour margin.
\end{remark}

\section{Numerical experiments}
\label{sec:numerics}

This section illustrates the contour-count indicator field and its use for
visible-pole cluster extraction.  The experiments are designed to test five
questions: whether \(P_\rho\) behaves as an imaging functional, whether fixed
thresholding can extract visible pole clusters, how the result depends on the
contour radius and threshold, what persistent post-processing contributes, and
how the method behaves in visibility-limited configurations.

Unless otherwise stated, the reciprocal search region is
\[
    U_R=\{\lambda\in\mathbb C:1/R<|\lambda|<1\},
\]
and the indicator field is evaluated on a uniform Cartesian grid restricted to
the shrunk annulus \(U_R^\rho\).  The synthetic coefficient sequence has the
form
\[
    a_k
    =
    h_k+\sum_{j=1}^{N}c_j\lambda_j^k,
    \qquad k=0,1,\ldots ,
\]
with three reciprocal poles
\[
    \lambda_1=0.8,\qquad
    \lambda_2=0.620998-0.299976\,\ii,\qquad
    \lambda_3=0.501485+0.307463\,\ii .
\]
A holomorphic background and small complex coefficient noise are included.  The
testing family consists of several determinant orders and shifts, and
\(P_\rho\) is computed from the fraction of local contours for which the
argument-principle count equals one.

For localization accuracy, when the estimated and true numbers of visible
clusters agree, we use the maximum matching error
\[
    e_{\max}
    =
    \min_{\pi\in\mathfrak S_N}
    \max_{1\leq j\leq N}
    |\widehat\lambda_{\pi(j)}-\lambda_j|,
\]
where \(\mathfrak S_N\) denotes the set of permutations of
\(\{1,\ldots,N\}\).  When the numbers differ, we report the recovered
components and interpret the result as a visibility-limited cluster count.

\subsection{Indicator-field behavior}
\label{subsec:numerics-indicator}

Figure~\ref{fig:indicator-field} shows the contour-count indicator field
\(P_\rho(\lambda)\) with \(\rho=0.035\).  The field is concentrated near the
three true reciprocal poles, while most of the search annulus has value close
to zero.  This confirms that the one-zero contour decisions, after averaging
over determinant orders and shifts, produce a useful imaging functional on the
reciprocal pole plane.

The high-response sets are not point masses.  This is expected: a local circle
centered at \(\lambda\) gives a one-zero response whenever the disk
\(B(\lambda,\rho)\) contains one empirical determinant zero.  Thus each pole
generates a spatially extended response region whose width is controlled by
\(\rho\) and by the spread of the empirical determinant zeros.

\begin{figure}[t]
    \centering
    \includegraphics[width=0.72\textwidth]{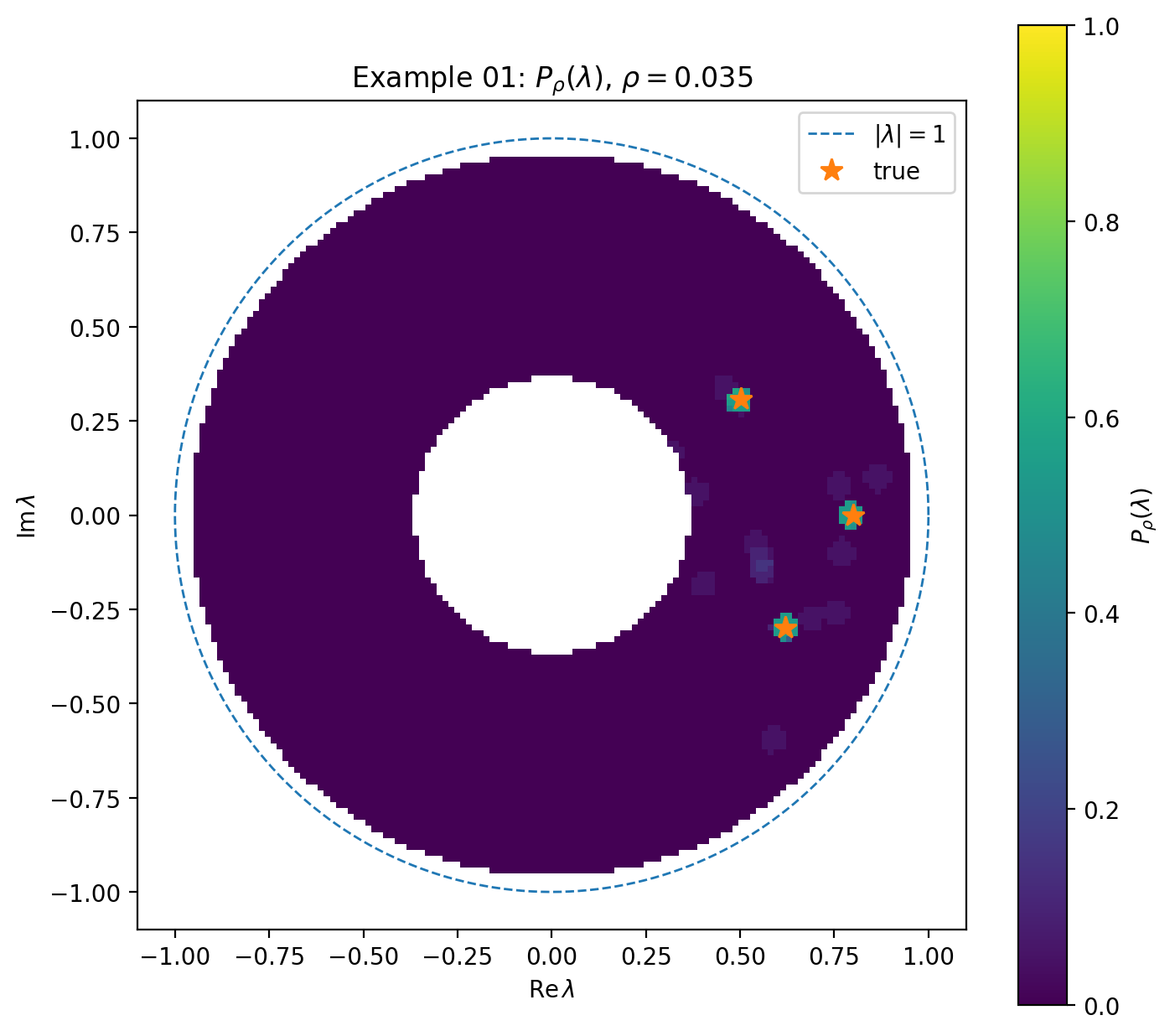}
    \caption{Contour-count indicator field \(P_\rho(\lambda)\) with
    \(\rho=0.035\).  The orange stars mark the true reciprocal poles.  The
    high-value regions concentrate near the true poles, while the background
    response is close to zero.}
    \label{fig:indicator-field}
\end{figure}

\subsection{Fixed-threshold cluster extraction}
\label{subsec:numerics-fixed-threshold}

We next apply the fixed-threshold extraction rule
\[
    \Omega_{\rho,\alpha}
    =
    \{\lambda\in U_R^\rho:P_\rho(\lambda)\geq\alpha\}.
\]
For \(\rho=0.035\) and \(\alpha=0.5\), the method extracts three connected
components.  Their weighted centers are listed in
Table~\ref{tab:fixed-threshold-components}.  The recovered centers are close to
the three true reciprocal poles, with nearest-pole errors between
\(5.08\times10^{-4}\) and \(4.20\times10^{-3}\).

\begin{table}[t]
\centering
\caption{Fixed-threshold components for \(\rho=0.035\) and
\(\alpha=0.5\).}
\label{tab:fixed-threshold-components}
\begin{tabular}{cccccc}
\toprule
Rank & Size & Max \(P_\rho\) & Center \(\widehat\lambda\) &
Median margin & Nearest error \\
\midrule
1 & 14 & 0.60 &
\(0.501763+0.308801\,\ii\) &
\(5.94\times10^{-9}\) &
\(1.37\times10^{-3}\) \\
2 & 15 & 0.55 &
\(0.624934-0.298516\,\ii\) &
\(4.49\times10^{-9}\) &
\(4.20\times10^{-3}\) \\
3 & 16 & 0.55 &
\(0.799492-0.000000\,\ii\) &
\(3.41\times10^{-9}\) &
\(5.08\times10^{-4}\) \\
\bottomrule
\end{tabular}
\end{table}

Figure~\ref{fig:threshold-clusters} displays the same indicator field together
with the extracted component centers.  The green circles overlap the true pole
locations, showing that the fixed-threshold rule already gives reliable
visible-pole localization in this example.

\begin{figure}[t]
    \centering
    \includegraphics[width=0.72\textwidth]{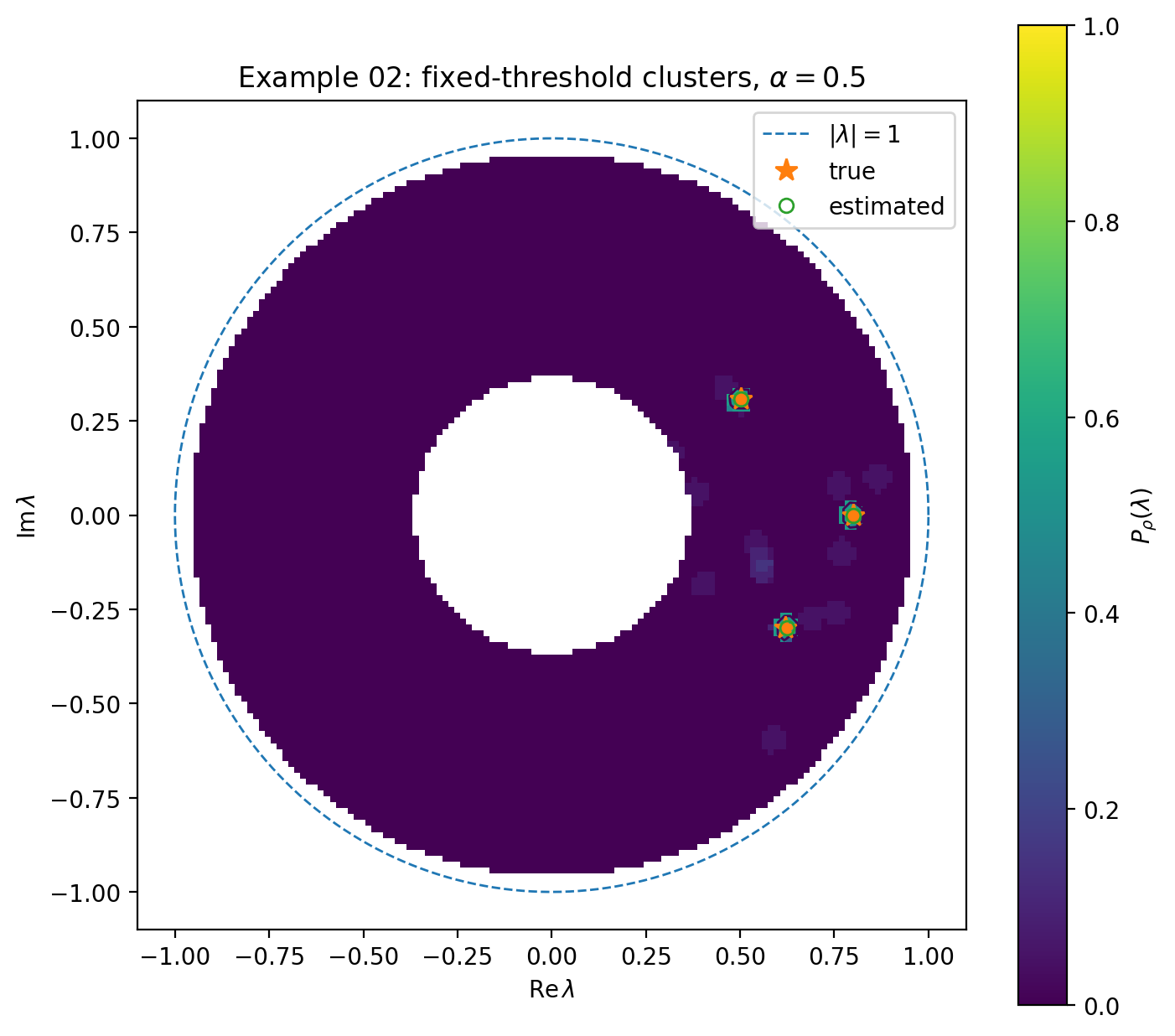}
    \caption{Fixed-threshold visible-pole clusters with \(\rho=0.035\) and
    \(\alpha=0.5\).  The orange stars are the true reciprocal poles and the
    green circles are the extracted weighted component centers.}
    \label{fig:threshold-clusters}
\end{figure}

\subsection{Dependence on the contour radius and threshold}
\label{subsec:numerics-rho-alpha}

The contour radius \(\rho\) controls the spatial scale of the indicator
response.  A small radius gives sharper responses but may fragment the field;
a larger radius gives smoother responses but may merge nearby features.
Figure~\ref{fig:rho-panel} shows \(P_\rho\) for
\[
    \rho=0.020,\ 0.035,\ 0.050,\ 0.070 .
\]
The visible components remain centered near the true reciprocal poles over this
range of radii, while the response regions broaden as \(\rho\) increases.

\begin{figure}[t]
    \centering
    \includegraphics[width=\textwidth]{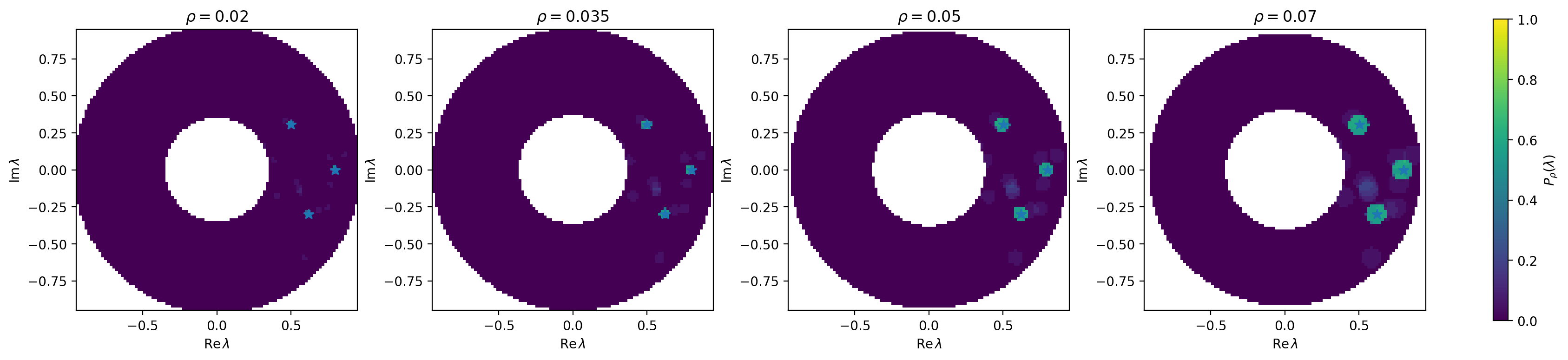}
    \caption{Indicator fields for different contour radii.  The pole responses
    remain stable over a range of \(\rho\), while larger radii produce broader
    high-response regions.}
    \label{fig:rho-panel}
\end{figure}

Table~\ref{tab:rho-alpha-sweep} reports the number of extracted components and
the matching error for different combinations of \(\rho\) and threshold
\(\alpha\).  For \(\alpha=0.35\) and \(\alpha=0.50\), the method extracts the
correct number of visible components for all tested radii.  The matching error
remains on the order of \(10^{-3}\).  For \(\alpha=0.65\), no component is
extracted because the threshold is above the observed peak level of the field
in most cases.  This illustrates why the threshold must be chosen relative to
the attainable indicator height.

\begin{table}[t]
\centering
\caption{Effect of contour radius \(\rho\) and threshold \(\alpha\).}
\label{tab:rho-alpha-sweep}
\begin{tabular}{ccccc}
\toprule
\(\rho\) & \(\alpha\) & No. components & \(e_{\max}\) & Max \(P_\rho\) \\
\midrule
0.020 & 0.35 & 3 & \(2.976\times10^{-3}\) & 0.55 \\
0.020 & 0.50 & 3 & \(2.976\times10^{-3}\) & 0.55 \\
0.020 & 0.65 & 0 & -- & 0.55 \\
0.035 & 0.35 & 3 & \(3.100\times10^{-3}\) & 0.60 \\
0.035 & 0.50 & 3 & \(4.745\times10^{-3}\) & 0.60 \\
0.035 & 0.65 & 0 & -- & 0.60 \\
0.050 & 0.35 & 3 & \(5.174\times10^{-3}\) & 0.60 \\
0.050 & 0.50 & 3 & \(5.174\times10^{-3}\) & 0.60 \\
0.050 & 0.65 & 0 & -- & 0.60 \\
0.070 & 0.35 & 3 & \(4.075\times10^{-3}\) & 0.65 \\
0.070 & 0.50 & 3 & \(4.302\times10^{-3}\) & 0.65 \\
0.070 & 0.65 & 0 & -- & 0.65 \\
\bottomrule
\end{tabular}
\end{table}

These results support the interpretation of \(P_\rho\) as a sampling-type
indicator.  The field gives stable visible-pole information over a moderate
range of contour radii and thresholds, but an overly high threshold may remove
all components.

\subsection{Persistent post-processing}
\label{subsec:numerics-persistence}

We then apply zero-dimensional persistence to the superlevel filtration of
\(P_\rho\).  This step is used as a post-processing tool, not as a replacement
for fixed-threshold extraction.  The corresponding persistence diagram is shown
in Figure~\ref{fig:persistence-diagram}.  The three dominant components have
birth levels approximately
\[
    0.60,\qquad 0.55,\qquad 0.55,
\]
and persist until the lowest background level in the discrete filtration.  In
contrast, many small components lie close to the diagonal or have short
lifetimes, reflecting low-level spurious responses.

\begin{figure}[t]
    \centering
    \includegraphics[width=0.62\textwidth]{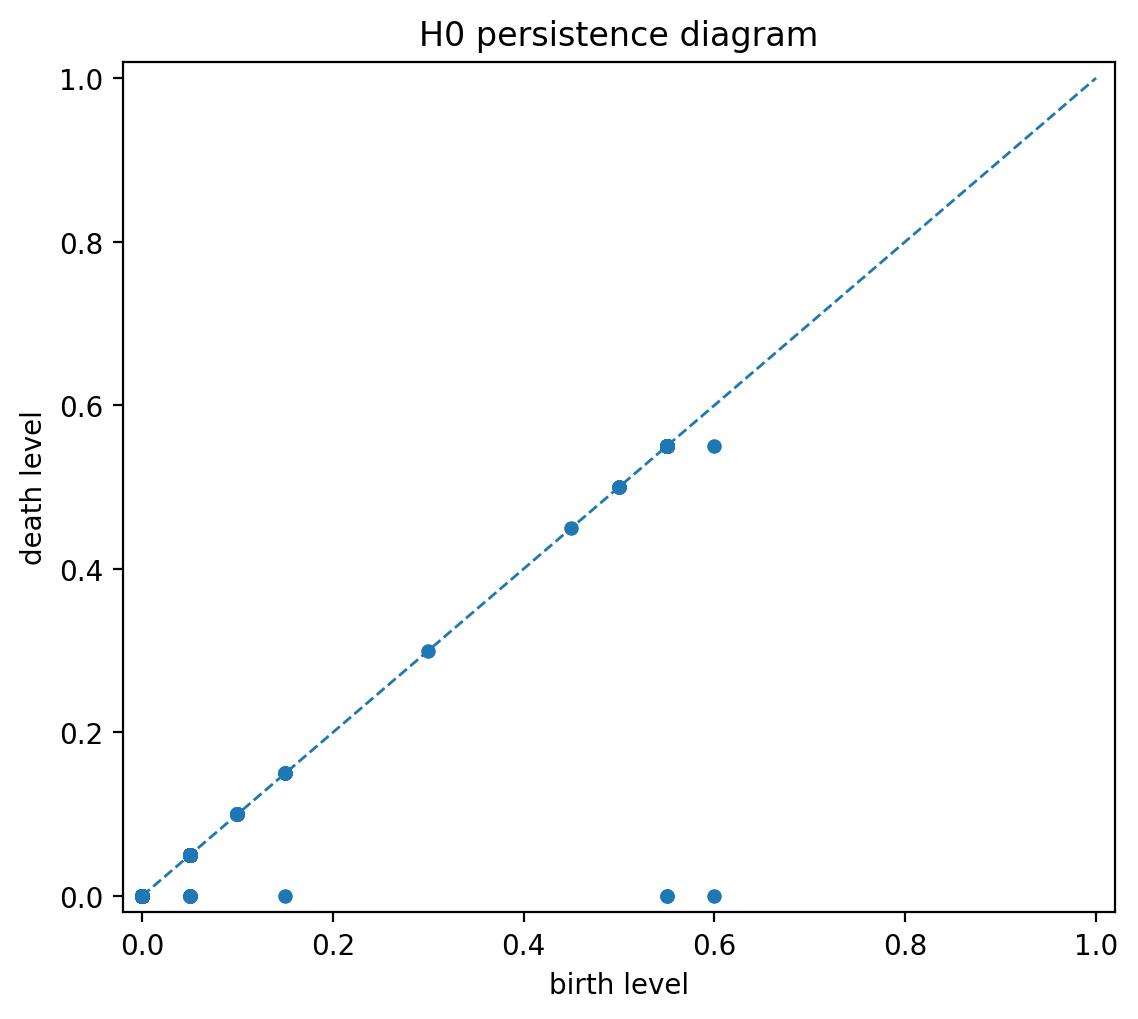}
    \caption{\(H_0\) persistence diagram of the superlevel filtration of
    \(P_\rho\).  The long-lived components correspond to the dominant visible
    pole clusters, while low-lifetime points represent short-lived fluctuations
    or low-level responses.}
    \label{fig:persistence-diagram}
\end{figure}

This experiment clarifies the role of persistent homology in the proposed
framework.  Fixed thresholding is sufficient when the high-response components
are clearly separated from the background.  Persistence becomes useful when one
wants to reduce dependence on a particular value of \(\alpha\) and retain only
components that survive over a range of thresholds.

\subsection{Visibility limitations}
\label{subsec:numerics-limitations}

Finally, we test three configurations where pole visibility is expected to
degrade.  The results are shown in Figure~\ref{fig:visibility-limitations} and
reported in Table~\ref{tab:visibility-limitations}.

\begin{figure}[t]
    \centering
    \includegraphics[width=\textwidth]{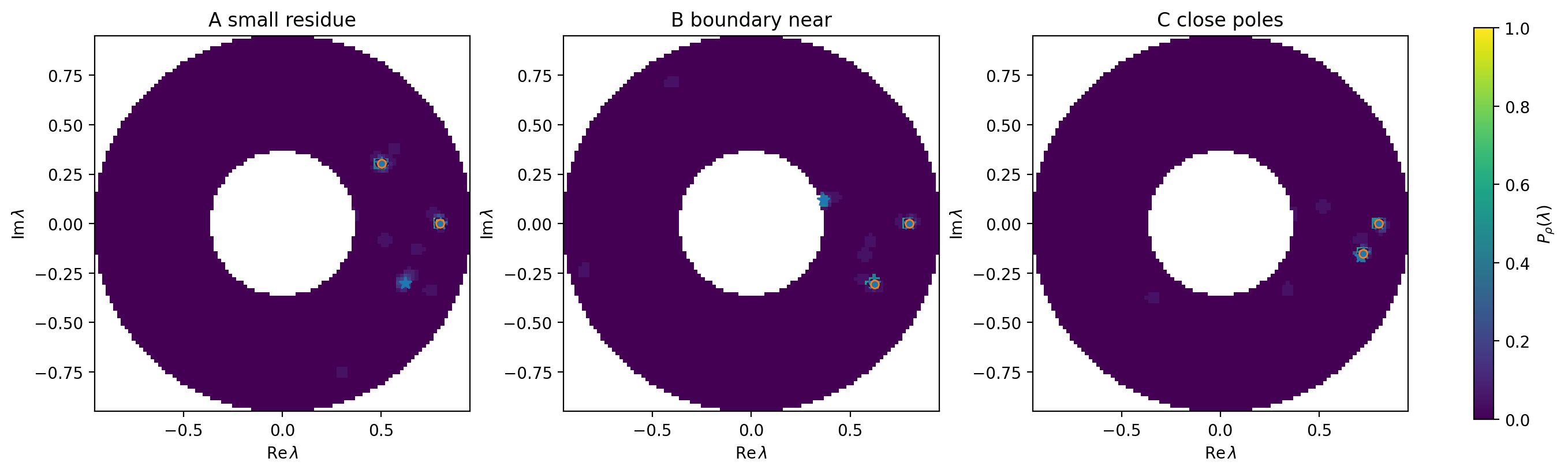}
    \caption{Visibility-limited examples.  Left: one pole has a small residue.
    Middle: one reciprocal pole is close to the boundary of the search annulus.
    Right: two reciprocal poles are close and merge into one visible cluster at
    the chosen scale.}
    \label{fig:visibility-limitations}
\end{figure}

\begin{table}[t]
\centering
\caption{Visibility-limited examples.  The method returns visible clusters,
not necessarily all formal poles.}
\label{tab:visibility-limitations}
\begin{tabular}{cccc}
\toprule
Case & No. components & Max \(P_\rho\) & Component centers \\
\midrule
Small residue & 2 & 0.75 &
\(0.798000+0.000934\,\ii;\ 0.501238+0.305448\,\ii\) \\
Boundary-near pole & 2 & 0.65 &
\(0.798178+0.000355\,\ii;\ 0.623301-0.304841\,\ii\) \\
Close poles & 2 & 0.65 &
\(0.722673-0.151664\,\ii;\ 0.799462-0.000731\,\ii\) \\
\bottomrule
\end{tabular}
\end{table}

In the small-residue case, the weak pole does not produce a sufficiently strong
high-response component at the chosen threshold.  In the boundary-near case,
the pole close to \(\partial U_R\) is difficult to certify because admissible
contours are geometrically restricted and the determinant margin is reduced.
In the close-pole case, two nearby reciprocal poles merge into a single visible
component.  These outcomes are consistent with the visible-cluster
interpretation of the method: the algorithm reports stable clusters supported
by the available contour-count evidence, rather than guaranteeing recovery of
every formal pole.

Overall, the experiments support the proposed workflow.  The contour-count
indicator field localizes visible reciprocal poles, fixed-threshold
superlevel components provide a direct extraction method, persistent homology
serves as a threshold-robust post-processing tool, and the observed failures
match the predicted visibility mechanisms.

\section{Conclusions}
\label{sec:conclusions}

We introduced a contour-count indicator field for visible pole clusters in outward meromorphic continuation.  The field aggregates local argument-principle zero counts over determinant orders and shifts and can be interpreted as a sampling-type imaging functional in the reciprocal pole plane.  Fixed superlevel sets provide direct visible-cluster extraction, while zero-dimensional persistent homology supplies an optional threshold-robust post-processing layer.

The analysis shows how pure-pole determinant factorization, one-zero contour geometry, Rouch\'e stability, superlevel-set separation, and persistence-gap stability combine to explain the observed behavior of the indicator field.  Strong isolated poles generate robust high-value components, while weak, close, boundary-near, or noise-dominated poles may produce low, short-lived, or merged responses.

Future work includes adaptive multiscale contour radii, sharper quantitative links between determinant margins and indicator-field contrast, and systematic comparisons with Hankel-pencil, Prony, Pad\'e, AAA, and inverse-scattering-inspired sampling indicators.

\bibliography{sn-bibliography}

@article{YangDeng2026,
	author = {X. Yang and Z. Deng},
	journal = {https://arxiv.org/abs/2607.04568},
	title = {Determinant characteristics and argument-principle certification for visible poles in meromorphic continuation},
	year = {2026}}

@article{Ito2012,
	author = {K. Ito and B. Jin and J. Zou},
	journal = {Inverse Problems},
	number = {2},
	pages = {025003},
	title = {A direct sampling method to an inverse medium scattering problem},
	volume = {28},
	year = {2012}}

@article{LiuSun2018,
	author = {J. Liu and J. Sun},
	journal = {Inverse Problems},
	number = {8},
	pages = {085007},
	title = {Extended sampling method in inverse scattering},
	volume = {34},
	year = {2018}}

@article{LiZou2013,
	author = {J. Z. Li and J. Zou},
	journal = {Inverse Problems and Imaging},
	number = {3},
	pages = {757-775},
	title = {A direct sampling method for inverse scattering using far-field data},
	volume = {7},
	year = {2013}}

@book{CakoniColtonMonk2011,
	publisher = {SIAM},
	title = {The Linear Sampling Method in Inverse Electromagnetic Scattering},
	year = {2011}}

@article{Potts2013,
	author = {D. Potts and M. Tasche},
	journal = {Electronic Transactions on Numerical Analysis},
	pages = {204-224},
	title = {Parameter estimation for multivariate exponential sums},
	volume = {40},
	year = {2013}}

@article{Sarkar1995,
	author = {T. K. Sarkar and O. Pereira},
	journal = {IEEE Antennas and Propagation Magazine},
	number = {1},
	pages = {48-55},
	title = {Using the matrix pencil method to estimate the parameters of a sum of complex exponentials},
	volume = {37},
	year = {1995}}

@book{Lavrentiev1967,
	address = {New York},
	author = {M. M. Lavrentiev},
	editor = {Translation revised by R. J. Sacker},
	publisher = {Springer-Verlag},
	title = {Some Improperly Posed Problems of Mathematical Physics},
	year = {1967}}

@article{Ying2022_b,
	author = {L. Ying},
	journal = {Journal of Scientific Computing},
	number = {3},
	pages = {107},
	title = {Pole recovery from noisy data on imaginary axis},
	volume = {92},
	year = {2022}}

@article{Deng2026_c,
	author = {Z. Deng and X. Guan and X. Yang},
	journal = {https://arxiv.org/abs/2606.29678},
	title = {Fourier--{Hankel} moment methods for topological counting and phase-center recovery in acoustic inverse scattering},
	year = {2026}}

@article{Deng2026_b,
	author = {Z. Deng and X. Yang and A. Qian},
	journal = {https://arxiv.org/pdf/2606.21815},
	title = {A moment-{Hankel} rank method for identifying the number of point sources in the heat equation},
	year = {2026}}

@article{Deng2026_a,
	author = {Z. Deng and A. Qian and X. Yang},
	journal = {https://arxiv.org/pdf/2606.15065},
	title = {A {Hankel} determinant zero-order principle for source conunting in an inverse heat point-source problem},
	year = {2026}}

@article{Ying2022,
	author = {L. Ying},
	journal = {Journal of Computational Physics},
	pages = {111549},
	title = {Analytic continuation from limited noisy {Matsubara} data},
	volume = {469},
	year = {2022}}

@article{Bisshopp1983,
	author = {F. Bisshopp},
	journal = {Quarterly of Applied Mathematics},
	number = {1},
	pages = {125-142},
	title = {Numerical conformal mapping and analytic continuation},
	volume = {41},
	year = {1983}}

@article{Stefanescu1980,
	author = {I. S. Stefanescu},
	journal = {Journal of Mathematical Physics},
	number = {1},
	pages = {175-188},
	title = {On the stable analytics continuation with rational functions},
	volume = {21},
	year = {1980}}

@article{Reichel1986,
	author = {L. Reichel},
	journal = {Constructive Approximation},
	pages = {23-39},
	title = {Numerical methods for analytic continuation and mesh generation},
	volume = {2},
	year = {1986}}

@article{Fu2009,
	author = {C. Fu and Z. Deng and X. Feng and F. Dou},
	journal = {SIAM Journal on Numerical Analysis},
	number = {4},
	pages = {2982-3000},
	title = {A modified {Tikhonov} regularization for stable analytic continuation},
	volume = {47},
	year = {2009}}

@article{DereviankoHuebner2025,
	author = {N. Derevianko and Lennart A. H{\"u}bner},
	journal = {https://arxiv.org/abs/2504.19157},
	title = {Parameter estimation for multivariate exponential sums via iterative rational approximation},
	year = {2025}}

@article{DereviankoPlonkaPetz2021,
	author = {Nadiia Derevianko and Gerlind Plonka and Markus Petz},
	journal = {https://arxiv.org/abs/2106.15140},
	title = {From ESPRIT to ESPIRA: Estimation of Signal Parameters by Iterative Rational Approximation},
	year = {2021}}

@article{DereviankoPlonka2022,
	author = {N. Derevianko and G. Plonka},
	journal = {Analysis and Applications},
	number = {3},
	pages = {543-577},
	title = {Exact reconstruction of extended exponential sums using rational approximation of their {Fourier} coefficients},
	volume = {20},
	year = {2022}}

@article{Derevianko2025,
	author = {N. Derevianko},
	journal = {Analysis and Appliations},
	title = {Recovery of rational functions via {Hankel} pencil method and sensitivities of the poles},
	year = {2025}}

@article{Henrici1966,
	author = {P. Henrici},
	journal = {Journal on Numerical Analysis},
	number = {1},
	pages = {67-78},
	title = {An algorithm for analytic continuation},
	volume = {3},
	year = {1966}}

@article{Cannon1965,
	author = {J. Cannon and K. Miller},
	journal = {Journal of the Society for Industrial and Applied Mathematics: Series B, Numerical Analysis},
	number = {1},
	pages = {87-98},
	title = {Some problems in numerical analytic continuation},
	volume = {2},
	year = {1965}}

@article{Miller1970,
	author = {K. Miller},
	journal = {SIAM Journal on Applied Mathematics},
	number = {2},
	pages = {346-363},
	title = {Stabilized numerical analytic prolongation with poles},
	volume = {18},
	year = {1970}}

@article{Trefethen2023,
	author = {L. N. Trefethen},
	journal = {Japan Journal of Industrial and Applied Mathematics},
	pages = {1587-1636},
	title = {Numerical analytic continuation},
	volume = {40},
	year = {2023}}

@article{Trefethen2020,
	author = {L. N. Trefethen},
	journal = {BIT Numerical Mathematics},
	pages = {901-915},
	title = {Quantifying the ill-conditioning of analytic continuation},
	volume = {60},
	year = {2020}}

\end{document}